\documentclass[11pt]{amsart}
\usepackage{amsmath,amsthm,amsfonts,amssymb,latexsym,epsfig}

\newtheorem{theorem}{Theorem}[section]

\newtheorem{lemma}{Lemma}[section]

\numberwithin{equation}{section}

\theoremstyle{definition}

\theoremstyle{remark}

\begin{document}
\title{On a result of Cartwright and Field}
\author{Peng Gao}
\address{Department of Mathematics, School of Mathematics and System Sciences, Beijing University of Aeronautics and Astronautics, P. R. China}
\email{penggao@buaa.edu.cn}
%%\date{\today}
%%\date{September 10, 2007.}
\thanks{The author is supported in part by NSFC grant 11371043.}

%%-------------------------------------------------------------------
\begin{abstract}  Let $M_{n,r}=(\sum_{i=1}^{n}q_ix_i^r)^{\frac {1}{r}}, r\neq 0$ and $M_{n,0}=\displaystyle \lim_{r \rightarrow 0}M_{n,r}$ be the weighted power means of $n$ non-negative numbers $x_i, 1 \leq i \leq n$ with $q_i > 0$ satisfying $\sum^n_{i=1}q_i=1$. Let $r>s$, a result of Cartwright and Field shows that when $r=1, s=0$,
\begin{align*}
    \frac {r-s}{2x_n}\sigma_n \leq M_{n,r}-M_{n,s} \leq \frac {r-s}{2x_1} \sigma_n,
\end{align*}
  where $x_1=\min \{ x_i \}, x_n=\max \{ x_i \}, \sigma_n=\sum_{i=1}^{n}q_i(x_i-M_{n,1})^2$.
In this paper, we determine all the pairs $(r,s)$ such that the right-hand side inequality above holds and all the pairs $(r,s), -1/2 \leq s \leq 1$ such that the left-hand side inequality above holds.
\end{abstract}

\maketitle
%%-----------------------------------------------------------------------
\section{Introduction}
%%------------------------------------------------------------------------

   Let $M_{n,r}({\bf x}; {\bf q})$ be the weighted power means:
   $M_{n,r}({\bf x}; {\bf q})=(\sum_{i=1}^{n}q_ix_i^r)^{\frac {1}{r}}$, where
   $M_{n,0}({\bf x}; {\bf q})$ denotes the limit of $M_{n,r}({\bf x}; {\bf q})$ as
   $r\rightarrow 0$, ${\bf x}=(x_1, \ldots,
   x_n)$, ${\bf q}=(q_1, \ldots,
   q_n)$ with $x_i \geq 0, q_i>0$ for all $1 \leq i \leq n$ and $\sum_{i=1}^nq_i=1$.  We further define $A_n({\bf x};{\bf q})=M_{n,1}({\bf x};{\bf q}), G_n({\bf x};{\bf q})=M_{n,0}({\bf x};{\bf q}), \sigma_n=\sum_{i=1}^{n}q_i(x_i-A_n)^2$. We shall write $M_{n,r}$ for $M_{n,r}({\bf x};{\bf q})$
   and similarly for other means when there is no risk of
   confusion.

   The following elegant refinement of the well-known arithmetic-geometric mean inequality is given by Cartwright and Field in \cite{C&F} :
\begin{align}
\label{1.5}
   \frac {\sigma_n}{2x_n}
    \leq   A_n-G_n  \leq  \frac {\sigma_n}{2x_1}.
\end{align}

   Naturally, one considers the following generalization of \eqref{1.5} on bounds for the differences of means:
\begin{align}
\label{1.5'}
    \frac {r-s}{2x_n}\sigma_n \leq M_{n,r}-M_{n,s} \leq \frac {r-s}{2x_1} \sigma_n , \quad r>s.
\end{align}

    It is shown in \cite[Theorem 3.2]{G4} that when $r = 1$ (resp. $s=1$), inequalities \eqref{1.5'} hold if and only if $-1 \leq s <1$ (resp. $1<r \leq 2$).
    Moreover, it is shown in \cite{G4} that the constant $(r-s)/2$ is best possible when either inequality in \eqref{1.5'} is valid. However, neither inequality in \eqref{1.5'} is valid for all $r,s$ and a necessary condition on $r, s$ such that either inequality of \eqref{1.5'} is valid is given in Lemma \ref{lem1} in Section \ref{sec 2}.

    In this paper, we determine all the pairs $(r,s)$ such that the right-hand side inequality of \eqref{1.5'} holds and on all the pairs $(r,s), -1/2 \leq s \leq 1$ such that the left-hand side inequality of \eqref{1.5'} holds. We prove in Section \ref{sec 3} the following
\begin{theorem}
\label{thm2}
  Let $r>s$ and $x_1=\min \{ x_i \}, x_n=\max \{ x_i \}$. The right-hand side inequality of \eqref{1.5'} holds if and only if $0 \leq r+s \leq 3$, $r \leq 2, s \geq -1$. When $-1/2 \leq s \leq 1$, the left-hand side inequality of \eqref{1.5'} holds if and only if $0 \leq r+s \leq 3$, $r \geq 1$.
  Moreover, in all these cases we have equality holding if and only if $x_1=x_2=\cdots=x_n$.
\end{theorem}
%%----------------------------------------------------------------------------
\section{Lemmas}
\label{sec 2} \setcounter{equation}{0}
%%----------------------------------------------------------------------------
   Our first lemma gathers known results on inequalities \eqref{1.5'}.
\begin{lemma}
\label{lem0}
  Let $r>s$ and $x_1=\min \{ x_i \}, x_n=\max \{ x_i \}$. Both inequalities in \eqref{1.5'} hold when $1 \leq r \leq 2$, $-1 \leq s \leq 1$.
  The right-hand side inequality of \eqref{1.5'} holds for $s=0$ if and only if $0< r \leq 2$, the left-hand side inequality of \eqref{1.5'} holds for $s=0$ if and only if $1 \leq r \leq 3$. Moreover, in all these cases we have equality holding if and only if $x_1=x_2=\cdots=x_n$.
\end{lemma}
\begin{proof}
    As it is shown in \cite[Theorem 3.2]{G4} that both inequalities in \eqref{1.5'} are valid when $-1 \leq s <1=r$ and $s=1 < r \leq 2$, the first assertion of the lemma follows from the observation that when either inequality in \eqref{1.5'} is valid for $r>r'$ and $r'>s$, then it is valid for $r>s$. The second assertion of the lemma is \cite[Theorem 2]{G2016}. The cases for equalities also follow from
    \cite[Theorem 3.2]{G4} and \cite[Theorem 2]{G2016}.
\end{proof}

   We define
\begin{align}
\label{2.0}
   F(x_1, \cdots, x_n, q_1, \cdots, q_n)=M_{n,r}-M_{n,s}-\frac {r-s}2\sigma_n.
\end{align}

   It is easy to see that the right-hand side inequality of \eqref{1.5'} is equivalent to $F \leq 0$ for $1=x_1<x_2<\cdots <x_{n-1}<x_n$ and the left-hand side inequality of \eqref{1.5'} is equivalent to $F \geq 0$ for $0<x_1<x_2<\cdots <x_{n-1}<x_n =1 $. We expect the extreme values of $F$ to occur at $n=2$ with one of the $x_i$ or $q_i$ taking a boundary value. Based on this consideration, we prove the following necessary condition for inequalities \eqref{1.5'} to hold.
\begin{lemma}
\label{lem1}
   Let $r>s \neq 0$. A necessary condition for the right-hand side inequality of \eqref{1.5'} to hold is that $0 \leq r+s \leq 3$, $r \leq 2, s \geq -1$.
   A necessary condition for the left-hand side inequality of \eqref{1.5'} to hold is that $0 \leq r+s \leq 3$, $r \geq 1, rs \leq 2$ and
\begin{align}
\label{2.2}
   \frac {r-s}{2} \leq (2-\frac 1r)^{2-\frac 1r}(1-\frac 1r)^{-(1-\frac 1r)},
\end{align}
   when $s<0$, where we define $0^0=1$.
\end{lemma}
\begin{proof}
Note first that it is shown in \cite[Lemma 3.1]{G4} that a necessary condition for either inequality of \eqref{1.5'} to hold is that $0 \leq r+s \leq 3$.
   Now we let $n=2, x_1=x, x_2=1, q_1=q$ and $F$ be defined as in \eqref{2.0} to see that,
\begin{align*}
  \lim_{q \rightarrow 0^+}\frac {F(x,1, q, 1-q)}{q} &=\frac {x^r-1}{r}-\frac {x^s-1}{s}-\frac {r-s}{2}(x-1)^2, \\
   \lim_{q \rightarrow 1^-}\frac {F(x,1, q, 1-q)}{1-q} &=\frac {x-x^{1-s}}{s}-\frac {x-x^{1-r}}{r}-\frac {r-s}{2}(x-1)^2.
\end{align*}
    As the first (second) right-hand side expression above is positive when $r>2$ ($s<-1$) and $x \rightarrow +\infty$, we conclude that in order for the right-hand side inequality of \eqref{1.5'} to hold, it is necessary to have $r \leq 2$ and $s \geq -1$. Moreover, the first (second) right-hand side expression above is negative when $s>0, rs>2$ ($r<1$) and $x=0$, we then conclude that in order for the left-hand side inequality of \eqref{1.5'} to hold, it is necessary to have $r \geq 1$ and $rs \leq 2$ (note that when $s < 0$, this condition is also satisfied).

    On the other hand, when $s<0$, we have
\begin{align*}
  \lim_{x \rightarrow 0^+}F(x,1, q, 1-q) =(1-q)^{1/r}-\frac {(r-s)q(1-q)}{2}.
\end{align*}
    In order for the left-hand side inequality of \eqref{1.5'} to hold for $s<0$, the expression above needs to be non-negative. On setting $y=1-q$, we see that this is equivalent to showing
\begin{align}
\label{2.22}
    y^{1/r-1}-\frac {(r-s)(1-y)}{2}
\end{align}
    is non-negative for $0 < y \leq 1$. As \eqref{2.2} implies that $s \geq -1$ when $r=1$, a condition already obtained in the discussions above, we may further assume that $r>1$. The expression in \eqref{2.22} is minimized at $y=(2(1-1/r)/(r-s))^{1/(2-1/r)}$ as one checks that this value is in between $0$ and $1$ when $r \geq 1$. Substituting this value in \eqref{2.22}, one checks easily that it is necessary to have \eqref{2.2} in order for the expression in \eqref{2.22} to be non-negative for $0<y \leq 1$ and the assertion of the lemma now follows.
\end{proof}

    We remark here that inequality \eqref{2.2} implies that it is not possible for the left-hand side inequality of \eqref{1.5'} to hold for $r>1$ and all $s<0$. In fact, by setting $z=1-1/r$, one checks easily that the right-hand side expression in \eqref{2.2} is an increasing function of $z$, hence is maximized at $z=1$, with value $4$. It follows then from \eqref{2.2} and the condition $r+s \geq 0$ that in order for the left-hand side inequality of \eqref{1.5'} to hold, it is necessary to have $4 \geq (r-s)/2 \geq (-s-s)/2=-s$, which implies that $s \geq -4$.

\begin{lemma}
\label{lem3}
   Let $-1 \leq s<0, 0 \leq q \leq q_1 \leq 1, 0<x_0 <1, x_0 \leq x^{-s} \leq 1, $, then
\begin{align*}
   q+(1-q)x^{-s} \leq x^{\alpha_1}.
\end{align*}
   for any $\alpha_1 \geq 0$ satisfying $\alpha_1 \leq \alpha_0$, where
\begin{align*}
   \alpha_0 =\frac {-s\ln ((1-q_1)x_0+q_1)}{\ln x_0}.
\end{align*}
\end{lemma}
\begin{proof}
   We let $y =x^{-s}$ and $\alpha=-\alpha_0/s$ so that $0 \leq \alpha <1$. It suffices to show that $g(y) \geq 0$ for $x_0 \leq y \leq 1$, where
\begin{align*}
   g(y)=y^{\alpha}-(q_1+(1-q_1)y).
\end{align*}
   It is easy to see that $g(y)$ is a concave function of $y$ and $g(1)=g( x_0) =0$,  hence the desired result follows.
\end{proof}

\begin{lemma}
\label{lem4}
   Let $-1 \leq s<0, 2< r \leq 3-s, 1-s^2 \leq (r-1)(r-2)$. Suppose that there exists a number $q_2, 1/2 \leq q_2 \leq 1$ such that
\begin{align*}
   \frac {\alpha_2}{s} :=\frac {(1-s^2)q_2}{(r-1)(r-2)(1-q_2)} \leq 1.
\end{align*}
   Then for $q_2 \leq q \leq 1$, $0<x \leq 1$,
\begin{align}
\label{2.4}
   (1-s)(q(1+s)x^s+(2-s)(1-q)) \geq x^{\alpha_2}(r-1)(-q(r+1)x^r+(r-2)(1-q)).
\end{align}
\end{lemma}
\begin{proof}
    As the expressions in \eqref{2.4} are linear functions of $q$, it suffices to prove inequality \eqref{2.4} for
    $q=q_2, 1$. The case $q=1$ is trivial and when $q=q_2$, we set $y=x^s$ to see that inequality \eqref{2.4} follows from $h(y) \geq 0$ for $y \geq 1$, where
\begin{align*}
   h(y)=(1-s)(q_2(1+s)y+(2-s)(1-q_2))-y^{\alpha_2/s}(r-1)(r-2)(1-q_2).
\end{align*}
   As $\alpha_2/s \leq 1$, it is easy to see that $h(y)$ is minimized at $y=1$ with a positive value and this completes the proof.
\end{proof}

   For $r>s, r^2+s^2 \neq 0, 0 \leq q<1, x>0$, we define
\begin{align}
\label{3.1}
   F_1(x, q) &=(qx^r+1-q)^{(1-r)/r}x^{r-1}-(qx^s+1-q)^{(1-s)/s}x^{s-1}-(r-s)(1-q)(x-1), \\
\label{3.2}
   F_2(x,q) & =(r-1)(q+(1-q)x^{-r})^{\frac {1-2r}r}x^{-r-1}+(1-s)(q+(1-q)x^{-s})^{\frac {1-2s}s}x^{-s-1}-(r-s).
\end{align}

   Part of our proof of Theorem \ref{thm2} needs $F_2(x,q) \geq 0$ for $0 < x \leq 1$ and various $q$. The following lemma gives a sufficient condition for this.
\begin{lemma}
\label{lem5}
   Let $-1 < s<0, 2 < r \leq 3-s, 0<x \leq 1, 0 \leq a, b<1$. If for $a \leq q < b$,
\begin{align}
\label{2.6}
    q+(1-q)x^{-s} & \leq x^{\alpha_1} ,   \\
(1-s)(q(1+s)x^s+(2-s)(1-q)) & \geq x^{\alpha_2}(r-1)(-q(r+1)x^r+(r-2)(1-q)). \nonumber
\end{align}
    Then $F_2(x,q) \geq 0$ for $a \leq q< b$ when $c(r,s, \alpha_1, \alpha_2) \leq 0$, where $F_2(x,q)$ is defined in
    \eqref{3.2} and
\begin{align}
\label{2.07}
   c(r,s, \alpha_1, \alpha_2) =& c_0(r,s)+((r-1)(2r-1)(1-3s)+(3r-1)(1-2s)(1-s))\frac {\alpha_1}{s}   \\
   &+(2r-1)(r-1)\alpha_2, \nonumber \\
   c_0(r,s) =& r^3-(5+4s)r^2+(2+6s+3s^2)r-s(2+s). \nonumber
\end{align}
\end{lemma}
\begin{proof}
    As in this case $\displaystyle \lim_{x \rightarrow 0^+}F_2(x,q)>0, F_2(1,q)=0$, we only need to show the values of $F_2$ at points satisfying:
\begin{align*}
   \frac {\partial F_2}{\partial x}=0,
\end{align*}
  are non-negative.

   Calculation shows that at these points, we have
\begin{align}
\label{2.3'}
   & (qx^r+1-q)^{\frac {1-3r}{r}}x^{r-2}   \\
 =& \frac {(1-s)(q(1+s)x^s+(2-s)(1-q))}{(r-1)(-q(r+1)x^r+(r-2)(1-q))}(q+(1-q)x^{-s})^{\frac {1-3s}{s}}x^{-1-2s}.  \nonumber
\end{align}

  If $-q(r+1)x^r+(r-2)(1-q) \leq 0$, then no such points exist. Hence we may assume that $-q(r+1)x^r+(r-2)(1-q) > 0$.
  Applying \eqref{2.6} in \eqref{2.3'}, we find that
\begin{align*}
   (qx^r+1-q)^{\frac {1-3r}{r}} \geq x^{\frac {1-3s}{s}\alpha_1+\alpha_2+1-2s-r}.
\end{align*}

     We write $\alpha_3=\frac {1-3s}{s}\alpha_1+\alpha_2$ so that the above inequality implies that
\begin{align*}
     qx^r+1-q \leq x^{\frac {r(\alpha_3+1-2s-r)}{1-3r}}.
\end{align*}

   We now apply the arithmetic-geometric inequality and the above estimation to see that
\begin{align*}
    & \frac {F_2(x,q)}{r-s}  \\
   =& \frac {r-1}{r-s}(qx^r+1-q)^{\frac {1-2r}r}x^{r-2}+\frac {1-s}{r-s}(q+(1-q)x^{-s})^{\frac {1-2s}s}x^{-1-s} -1 \nonumber \\
   \geq &\frac {r-1}{r-s}x^{(\alpha_3+1-2s-r)(1-2r)/(1-3r)}x^{r-2}+\frac {1-s}{r-s}x^{(1-2s)\alpha_1/s}x^{-1-s} -1 \nonumber \\
    \geq & \displaystyle x^{(\alpha_3+1-2s-r)(1-2r)(r-1)/((1-3r)(r-s))} \cdot x^{((r-1)(r-2)-(1-s^2))/(r-s)} \cdot x^{(1-2s)\alpha_1(1-s)/(s(r-s))} -1 \nonumber \\
   = & \displaystyle x^{c(r,s, \alpha_1, \alpha_2)/((r-s)(3r-1))}-1. \nonumber
\end{align*}
   The assertion of the lemma now follows easily.
\end{proof}

\begin{lemma}
\label{lem2}
   Let $-1 < s<0, 2 < r \leq 3-s $. Let $c(r,s,\alpha_1, \alpha_2)$ be defined as in Lemma \ref{lem5}.
   Define
\begin{align}
\label{1.3}
  c_1(r,s) &=r^3-(6+s)r^2+(s^2+4)r-s(s^2-6s+4), \\
  c_2(r,s) &=(r-1)(-1-2s)-(1-s)(1+s), \nonumber \\
  c_3(r,s) &=c(r,s, 0, s),  \nonumber \\
  c_4(r,s) &=c(r,s, -0.0889s, \frac {(1-s^2)s}{(r-1)(r-2)}). \nonumber
\end{align}
  Then $\displaystyle \max_{1 \leq i \leq 4} \{ c_i(r,s)\} \leq 0$ when $-1/2 \leq s<0$.
\end{lemma}
\begin{proof}
   It is easy to see that $c_i, 0 \leq i \leq 3$ are all convex functions of $r \geq 2$, where $c_0(r,s)$ is defined in \eqref{2.07}.  Also, $c_4$ is a convex function of $r \geq 3$.  Thus, it suffices to show that $c_i, 1 \leq i \leq 3$ are non-positive for $-1/2 \leq s<0, r=2, 3-s$ and that $c_4$ is non-positive for $-1/2 \leq s<0, 2 < r \leq 3, r=3-s$.    One checks directly that $\displaystyle \max_{i=1,2} \{ c_i(2,s), c_i(3-s,s) \} \leq 0, c_3(2,s) \leq c_0(2,s) \leq 0$ and that $\displaystyle \max_{2 <r \leq 3} c_4(r,s) \leq \displaystyle\max_{2 <r \leq 3} c_0(r,s)  \leq \displaystyle\max \{ c_0(2,s), c_0(3,s) \} =0$. We also have
\begin{align*}
      c_3(3-s,s) &=-12-9s +21s^2-6s^3, \\
      c_4(3-s,s) & = -12-14s +33s^2-10s^3-0.0889(18-66s+54s^2-12s^3).
\end{align*}

   As both expressions on the right-hand side above are decreasing functions of $s<0$, and one checks directly that $c_3(3-s,s)<0, c_4(3-s,s) <0$ for $s=-1/2$, it follows that $\displaystyle \max_{i=3,4} \{ c_i(3-s,s) \} \leq 0$ and this completes the proof.

\end{proof}
%%--------------------------------------------------------------------------------------
%%--------------------------------------------------------------------------------------

%%----------------------------------------------------------------------------
\section{Proof of Theorem \ref{thm2}}
\label{sec 3} \setcounter{equation}{0}
%%----------------------------------------------------------------------------

     We assume that $r>s$ throughout this section. We omit the discussions on the conditions for equality in each inequality we shall prove as one checks easily that the desired conditions hold by going through our arguments in what follows.  As the case $s=0,1$ or $r=1$ has been proven in \cite[Theorem 3.2]{G4} and \cite[Theorem 2]{G2016}, we further assume $r \neq 1, s \neq 0,1$ in what follows.

     Now, Lemma \ref{lem1} implies that it remains to prove the ``if " part of Theorem \ref{thm2}. We consider the right-hand side inequality of \eqref{1.5'} first.
     Let $F$ be defined as in \eqref{2.0} and we assume that $x_1=1<x_2<\cdots <x_n$, $q_i>0, 1 \leq i \leq n$.
    We have
\begin{align*}
   F_0(x_1, \cdots, x_n, q_1, \cdots, q_n): =\frac {\partial F}{q_n\partial x_n}=M^{1-r}_{n,r}x^{r-1}_n-M^{1-s}_{n,s}x^{s-1}_n-(r-s)(x_n-A_n).
\end{align*}
   Now the right-hand side inequality of \eqref{1.5'} follows from $F \leq 0$, which in turn follows from $F_0 \leq 0$ as it implies $F({\bf x};{\bf q} ) \leq \lim_{x_n \rightarrow x_{n-1}}F({\bf x};{\bf q})$. By adjusting the value of $q_{n-1}$ in the expression of $\lim_{x_n \rightarrow x_{n-1}}F({\bf x};{\bf q})$ and repeating the process, it follows easily that $F \leq 0$.

   When $n \geq 3$, we regard $x_1=1, x_n$ as fixed and assume that $F_0$ is maximized at some point $({\bf x}';{\bf q}')=(x'_1, \cdots, x'_n, q'_1, \cdots, q'_n)$ with $x'_1=x_1, x'_n=x_n$. Then at this point we must have
\begin{align*}
   \frac {\partial F_0}{\partial x_i}\Big |_{({\bf x}';{\bf q}')}=0, \quad 2 \leq i \leq n-1.
\end{align*}
   Thus, the $x'_i, 2 \leq i \leq n-1$ are solutions of the equation:
\begin{align*}
   f_1(x):=(1-r)M^{1-2r}_{n,r}x^{r-1}_nx^{r-1}-(1-s)M^{1-2s}_{n,s}x^{s-1}_nx^{s-1}+r-s=0.
\end{align*}
   It is easy to see that the above equation can have at most two different positive roots.

   On the other hand, by applying the method of Lagrange multipliers, we let
\begin{align*}
   \tilde{F}_0(x_1, \cdots, x_n, q_1, \cdots, q_n, \lambda)=F_0(x_1, \cdots, x_n, q_1, \cdots, q_n)-\lambda(\sum^n_{i=1}q_i-1),
\end{align*}
   where $\lambda$ is a constant. Then at $({\bf x}';{\bf q}')$ we must have
\begin{align*}
   \frac {\partial \tilde{F}_0}{\partial q_i}\Big |_{({\bf x}';{\bf q}')}=0, \quad 1 \leq i \leq n.
\end{align*}
   Thus, the $x'_i, 1 \leq i \leq n$ are solutions of the equation:
\begin{align*}
   f_2(x) := \frac {1-r}{r}M^{1-2r}_{n,r}x^{r-1}_nx^{r}-\frac {1-s}{s}M^{1-2s}_{n,s}x^{s-1}_nx^{s}+(r-s)x-\lambda=0.
\end{align*}
   As $f'_2(x)=f_1(x)$, it follows from the mean value theorem that there is a solution of $f_1(x)=0$ between any two adjacent $x'_i, x'_{i+1}, 1 \leq i \leq n-1$, as they are solutions of $f_2(x)=0$.  But when $n \geq 3$, we have at least $x'_2$ as a solution of $f_1(x)=0$. This would imply that
   $f_1(x)=0$ has at least three different positive solutions (for example, one in between $x'_1$ and $x'_2$, one in between $x'_2$ and $x'_3$, and $x'_2$ itself), a contradiction.

   Therefore, it remains to show $F_0 \leq 0$ for $n=2$. In this case, we let $1=x_1<x_2=x, 0< q_2=q <1, q_1=1-q$ to see that $F_0=F_1(x,q)$, where $F_1(x,q)$ is defined in \eqref{3.1}.

    Note that $F_2(x,q) = (1-q)^{-1}\partial F_1/\partial x$, where $F_2(x,q)$ is defined in \eqref{3.2}. As $F_1(1,q)=0$, we see that it suffices to show that $F_2(x,q) \leq 0$ for $x \geq 1$.

   We now divide the proof of the right-hand side inequality of \eqref{1.5'} for $r>s \neq 0$ satisfying $0 \leq r+s \leq 3$, $r \leq 2, s \geq -1$ into several cases. As the case $-1 \leq s \leq 1 \leq r \leq 2$ follows directly from Lemma \ref{lem0}, we only consider the remaining cases in what follows and we show in these cases $F_2(x,q) \leq 0$ or equivalently, $F_2(x,q)/(r-s)+1 \leq 1$. Note that
\begin{align}
\label{3.2'}
    \frac {F_2(x,q)}{r-s}+1 = \frac {r-1}{r-s}(q+(1-q)x^{-r})^{\frac {1-2r}r}x^{-r-1}+\frac {1-s}{r-s}(q+(1-q)x^{-s})^{\frac {1-2s}s}x^{-s-1}.
\end{align}
   One checks that in all the following cases, we have $(r-1)(1-s) \leq 0$. Therefore, it follows from the arithmetic-geometric mean inequality with non-positive weights that the right-hand side expression in \eqref{3.2'} is less than or equal to
\begin{align}
\label{2.1}
  &  (q+(1-q)x^{-r})^{(1-2r)(r-1)/(r(r-s))}(q+(1-q)x^{-s})^{(1-2s)(1-s)/(s(r-s))}x^{(s^2-r^2)/(r-s)} \\
  =& (qx^r+1-q)^{(1-2r)(r-1)/(r(r-s))}(qx^s+1-q)^{(1-2s)(1-s)/(s(r-s))}x^{r+s-3}. \nonumber
\end{align}

    Thus, it suffices to show that either side expression in \eqref{2.1} is $\leq 1$.

   Case 1. $0< s \leq 1/2 \leq r < 1$.

   Each factor of the left-hand side expression in \eqref{2.1} is $\leq 1$, hence their product is $\leq 1$.

   Case 2. $0<s<r \leq 1/2$.

   As it is well-known that $r \mapsto M_{n,r}$ is an increasing function of $r$ and $-r<-s$, we have
\begin{align*}
   (q+(1-q)x^{-r})^{-1/r} \leq (q+(1-q)x^{-s})^{-1/s}.
\end{align*}
   As we also have $(1-2r)(1-r) \leq (1-2s)(1-s)$, it follows that
\begin{align*}
   & (q+(1-q)x^{-r})^{(1-2r)(r-1)/(r(r-s))}(q+(1-q)x^{-s})^{(1-2s)(1-s)/(s(r-s))} \\
   \leq & (q+(1-q)x^{-s})^{((1-2r)(1-r)-(1-2s)(1-s))/(-s(r-s))} \leq 1,
\end{align*}
   which implies that the left-hand side expression of \eqref{2.1} is $\leq 1$.

   Case 3.  $1/2 \leq s<r <1$.

    Note that
\begin{align*}
   q+(1-q)x^{-r} \leq q+(1-q)x^{-s},
\end{align*}
   so that the left-hand side expression of \eqref{2.1} is less than or equal to
\begin{align*}
   & (q+(1-q)x^{-s})^{(1-2r)(r-1)/(r(r-s))}(q+(1-q)x^{-s})^{(1-2s)(1-s)/(s(r-s))}x^{(s^2-r^2)/(r-s)} \\
   =& \left\{\begin{array}{l} (q+(1-q)x^{-s})^{1/(rs)-2}x^{(s^2-r^2)/(r-s)} \leq 1,  \quad \frac 1{rs} \geq 2, \\
    (qx^s+1-q)^{1/(rs)-2}x^{-1/r-r+s} \leq 1, \quad \frac 1{rs} \leq 2.
    \end{array}\right.
\end{align*}
 This implies that the left-hand side expression of \eqref{2.1} is $\leq 1$.

   Case 4. $1<s<r \leq 3-s=\min \{ 2, 3-s \}$.

    Note that
\begin{align*}
   qx^r+1-q \geq qx^s+1-q \geq 1, \quad (1-2r)(r-1) \leq 0.
\end{align*}

   It follows that
\begin{align*}
   & (qx^r+1-q)^{(1-2r)(r-1)/(r(r-s))}(qx^s+1-q)^{(1-2s)(1-s)/(s(r-s))}x^{r+s-3} \\
   \leq & (qx^s+1-q)^{(1-2r)(r-1)/(r(r-s))}(qx^s+1-q)^{(1-2s)(1-s)/(s(r-s))}x^{r+s-3} \\
   =& (qx^s+1-q)^{1/(rs)-2}x^{r+s-3}\leq 1,
\end{align*}
   which implies that the right-hand side expression of \eqref{2.1} is $\leq 1$.

   Case 5. $s<0<r<1, r+s \geq 0$.

    When $r \geq 1/2$, each factor of the left-hand side expression of \eqref{2.1} is $\leq 1$, hence their product is $\leq 1$. If $0<r <1/2$, then again it follows from the fact that $r \mapsto M_{n,r}$ is an increasing function of $r$ that
\begin{align*}
   (qx^r+1-q)^{1/r} \geq (qx^s+1-q)^{1/s} \geq 1.
\end{align*}
   As $(1-2r)(r-1) \leq 0$ and $3-2(r+s) \geq 0$, it follows that
\begin{align*}
   & (qx^r+1-q)^{(1-2r)(r-1)/(r(r-s))}(qx^s+1-q)^{(1-2s)(1-s)/(s(r-s))}x^{r+s-3} \\
   \leq & (qx^s+1-q)^{(1-2r)(r-1)/(s(r-s))}(qx^s+1-q)^{(1-2s)(1-s)/(s(r-s))}x^{r+s-3} \\
   =& (q+(1-q)x^{-s})^{(3-2(r+s))/s}x^{-(r+s)}\leq 1.
\end{align*}
   This now completes the proof for all the cases for the right-hand side inequality of \eqref{1.5'}.

%%-----------------------------------------------------------------------

  Next, we prove the left-hand side inequality of \eqref{1.5'} for $0 \leq r+s \leq 3$, $-1/2 \leq s \leq 1, r \geq 1$.  In this case, it suffices to show $F \geq 0$ provided that we assume $0<x_1<x_2<\cdots <x_n=1$. Similar to our discussions above, one shows easily that this follows from $\partial F/\partial x_1 \leq 0$ for $n=2$, which is equivalent to $F_1(x,q) \leq 0$ for $0<x \leq 1$.  Again we divide the proof into several cases. As the case $-1 \leq s \leq 1 \leq r \leq 2$ follows directly from Lemma \ref{lem0}, we only consider the remaining cases in what follows and similar to our proof of the right-hand side inequality of \eqref{1.5'} above, it suffices to show that $F_2(x,q) \geq 0$ for $0<x \leq 1$.

   Case 1. $1/2 \leq s<1,  2< r \leq 3-s$.

   As $r-1>0$, it follows from the arithmetic-geometric mean inequality that the right-hand side expression of \eqref{3.2'} is greater than or equal to the expressions in \eqref{2.1}. As the factors of the right-hand side expression of \eqref{2.1} are all $\geq 1$, it follows that $F_2(x,q) \geq 0$.

   For the remaining cases, one checks easily that we have $\lim_{x \rightarrow 0^+}F_2(x,q)>0, F_2(1,q)=0$ so that it suffices to show the values of $F_2(x,q)$ at points satisfying \eqref{2.3'} are non-negative, assuming that
$-q(r+1)x^r+(r-2)(1-q) > 0$. Hence, in what follows, we shall only evaluate $F_2(x,q)$ at these points satisfying the above assumption. We then note that at these points, we have
\begin{align}
\label{3.4'}
  (1-s)(q(1+s)x^s+(2-s)(1-q)) \geq (r-1)(-q(r+1)x^r+(r-2)(1-q)).
\end{align}
   This is seen by noting that the expressions in \eqref{3.4'} are linear functions of $q$, hence it suffices to check the validity of inequality \eqref{3.4'} at $q=0,1$.

   It then follows from \eqref{3.4'} and \eqref{2.3'} that at these points we have
\begin{align}
\label{3.3}
     (qx^r+1-q)^{\frac {1-3r}{r}}x^{r-2}\geq (q+(1-q)x^{-s})^{\frac {1-3s}{s}}x^{-1-2s}=(qx^s+1-q)^{\frac {1-3s}{s}}x^{s-2},
\end{align}
   an inequality we shall assume in what follows.

   Case 2. $0<s<1/2,  2< r \leq 3-s$.

   Similar to the previous case, the right-hand side expression of \eqref{3.2'} is greater than or equal to the expressions in \eqref{2.1}.
   From \eqref{3.3} we deduce that
\begin{align*}
     qx^r+1-q \leq  (qx^s+1-q)^{\frac {r(1-3s)}{s(1-3r)}}x^{\frac {r(s-r)}{1-3r}}.
\end{align*}
   Using this, we see that the right-hand side expression of \eqref{2.1} is greater than or equal to
\begin{align*}
   (qx^s+1-q)^{(2r-1)(r-1)(1-3s)/(s(3r-1)(r-s))}(qx^s+1-q)^{(1-2s)(1-s)/(s(r-s))}x^{-(2r-1)(r-1)/(3r-1)}x^{r+s-3}.
\end{align*}
    When $1/3 \leq s < 1/2$, we see that the first factor and the last factor above is $\geq 1$ and we write the product of the two factors in the middle as
\begin{align}
\label{3.4}
   (q+(1-q)x^{-s})^{(1-2s)(1-s)/(s(r-s))}x^{(1-2s)(1-s)/(r-s)-(2r-1)(r-1)/(3r-1)}.
\end{align}
    Note that the first factor above is now $\geq 1$ and it is easy to see that
\begin{align*}
   \frac {(1-2s)(1-s)}{r-s} \leq \frac {1-2s}{2-s} \leq \frac {1-2\cdot 1/3}{2-1/3} \leq \frac {2r-1}{3r-1} \leq \frac {(2r-1)(r-1)}{3r-1}.
\end{align*}
    This implies that the second factor in \eqref{3.4} is also $\geq 1$. Hence the right-hand side expression of \eqref{2.1} is greater than or equal to $1$ and it follows that $F_2(x,q) \geq 0$.

    When $0 < s < 1/3$,  it follows from \eqref{3.3} that
\begin{align}
\label{3.6}
     (qx^r+1-q)^{\frac {1-2r}{r}}x^{r-2} \geq (qx^r+1-q)(qx^s+1-q)^{\frac {1-3s}{s}}x^{s-2}.
\end{align}
   If the right-hand side expression above is $\geq 1$, then we have
\begin{align}
\label{3.7}
   F_2(x,q) &=(r-1)(qx^r+1-q)^{\frac {1-2r}r}x^{r-2}+(1-s)(q+(1-q)x^{-s})^{\frac {1-2s}s}x^{-s-1} -(r-s) \\
       & \geq (r-1)+(1-s) -(r-s)=0. \nonumber
\end{align}

   If the right-hand side expression of \eqref{3.6} is $\leq 1$,  then it implies that
\begin{align*}
   (qx^r+1-q) \leq (qx^s+1-q)^{-\frac {1-3s}{s}}x^{-(s-2)}.
\end{align*}
   Thus,
\begin{align*}
   (qx^r+1-q)^{\frac {1-2r}r}x^{r-2} & \geq ((qx^s+1-q)^{-\frac {1-3s}{s}}x^{-(s-2)})^{(1-2r)/r}x^{r-2} \\
   &=(q+(1-q)x^{-s})^{\frac {(1-3s)(2r-1)}{rs}}x^{\frac {(1+2s)(1-2r)}{r}+r-2} \geq 1 ,
\end{align*}
    where the last inequality above follows from the observation that the function $r \mapsto (1+2s)(1-2r)/r+r-2$ is an increasing function of $r \geq 2$ and hence is maximized at $r=3-s$, in which case its value is easily shown to be negative. It follows from \eqref{3.7} that  $F_2(x,q) \geq 0$ in this case.

   Case 3. $-1/2 \leq s<0,  2< r \leq 3-s$.

  We divide this case into a few subcases:

Subcase 1. $0<q \leq 1/2$.

   As $r \mapsto M_{n,r}$ is an increasing function of $r$ and $-s \leq r$ since $r+s \geq 0$, we have
\begin{align}
\label{2.08}
     (q+(1-q)x^{-s})^{-1/s} \leq (q+(1-q)x^{r})^{1/r} .
\end{align}
    As $0< q \leq 1/2$, we also have
\begin{align}
\label{3.10'}
     q+(1-q)x^{r} \leq qx^r+1-q.
\end{align}

    We then deduce from \eqref{3.3}, \eqref{2.08} and \eqref{3.10'} that
\begin{align*}
     (qx^r+1-q)^{\frac {1-3r}{r}}x^{r-2} \geq (q+(1-q)x^{-s})^{\frac {1-3s}{s}}x^{-1-2s} \geq (qx^r+1-q)^{-\frac {1-3s}{r}}x^{-1-2s}.
\end{align*}
    It follows that (note that $2-3(r+s) \leq 0$ for $r \geq 2, s \geq -1$)
\begin{align}
\label{3.11}
     qx^r+1-q \leq x^{\frac {r(1-2s-r)}{2-3(r+s)}}.
\end{align}

   With $c_1(r,s)$ being defined in \eqref{1.3}, we then deduce that
\begin{align*}
   & \frac {F_2(x,q)}{r-s} \\
   =&  \frac {r-1}{r-s}(qx^r+1-q)^{\frac {1-2r}r}x^{r-2}+\frac {1-s}{r-s}(q+(1-q)x^{-s})^{\frac {1-2s}s}x^{-1-s}-1  \nonumber \\
   \geq &  \frac {r-1}{r-s}(qx^r+1-q)^{\frac {1-2r}r}x^{r-2}+\frac {1-s}{r-s}(qx^r+1-q)^{-\frac {1-2s}r}x^{-1-s}-1  \nonumber \\
   \geq & (qx^r+1-q)^{((1-2r)(r-1)-(1-s)(1-2s))/(r(r-s))}\cdot x^{((r-1)(r-2)-(1-s)(1+s))/(r-s)} -1 \nonumber \\
   \geq & \displaystyle x^{(1-2s-r)((1-2r)(r-1)-(1-s)(1-2s))/((2-3(r+s))(r-s))}\cdot x^{((r-1)(r-2)-(1-s)(1+s))/(r-s)}-1   \nonumber \\
   = & x^{c_1(r,s)/((r-s)(3(r+s)-2))}-1, \nonumber
\end{align*}
   where the first inequality above follows from \eqref{2.08} and \eqref{3.10'}, the second inequality above follows from the arithmetic-geometric inequality and the last inequality above follows from \eqref{3.11}. It follows from Lemma \ref{lem2} that $F_2(x,q) \geq 0$ in this case.

%%--------------------------------------------------------------------------------------
%%--------------------------------------------------------------------------------------

   Subcase 2. $1/2 \leq q \leq 1, (1+s)x^s-(2-s) \geq 0$ or $1-s^2 \geq (r-1)(r-2)$.

   One checks that if $(1+s)x^s-(2-s) \geq 0$, then the function $q \mapsto (1-s)(q(1+s)x^s+(2-s)(1-q))(qx^r+1-q)-(r-1)(-q(r+1)x^r+(r-2)(1-q))$ is a concave function of $q$ and hence is minimized at $q=0,1$, with values $\geq 1$.

   If $1-s^2 \geq (r-1)(r-2)$, then
\begin{align*}
   (1-s)(q(1+s)x^s+(2-s)(1-q))\geq  q(1-s)(1+s)+(1-s)(2-s)(1-q) \geq  (r-1)(r-2).
\end{align*}

   It follows that
\begin{align*}
   & (1-s)(q(1+s)x^s+(2-s)(1-q))(qx^r+1-q) \\
    \geq & (r-1)(r-2)(1-q) \\
   \geq & (r-1)(-q(r+1)x^r+(r-2)(1-q)).
\end{align*}

   Thus, in either case, we deduce from the above and \eqref{2.3'} that we have
\begin{align*}
   (qx^r+1-q)^{\frac {1-2r}{r}}x^{r-2} \geq (q+(1-q)x^{-s})^{\frac {1-3s}{s}}x^{-1-2s} \geq x^{-1-2s}.
\end{align*}
   From this we apply the arithmetic-geometric inequality to see that
\begin{align}
\label{2.5}
   \frac {F_2(x,q)}{r-s}& =\frac {r-1}{r-s}(qx^r+1-q)^{\frac {1-2r}r}x^{r-2}+\frac {1-s}{r-s}(q+(1-q)x^{-s})^{\frac {1-2s}s}x^{-1-s} -1 \\
   &  \geq \frac {r-1}{r-s}x^{-1-2s}+\frac {1-s}{r-s}x^{-1-s} -1 \nonumber \\
   & \geq x^{\frac {c_2(r,s)}{r-s}}  -1 \nonumber.
\end{align}
   where $c_2(r,s)$ is defined in \eqref{1.3}. Now Lemma \ref{lem2} implies that $F_2(x,q) \geq 0$ in this case.

   Subcase 3. $(1+s)x^s-(2-s) \leq 0, 1-s^2 \leq (r-1)(r-2)$ and $1/2 \leq q_0  \leq q<1$, where $q_0$ is defined by
\begin{align}
\label{3.10}
   \frac {(1-s^2)q_0}{(r-1)(r-2)(1-q_0)}=1.
\end{align}

    In this case, Lemma \ref{lem4} with $q_2=q_0$ implies that \eqref{2.6} is satisfied by $\alpha_1=0$ and
    $\alpha_2=s$, where we set $a=q_0$ and $b=1$ in Lemma \ref{lem5}. It follows from Lemma \ref{lem5} that $F_2(x,q) \geq 0$ as long as $c_3(r,s) \leq 0$, where $c_3(r,s)$ is given in \eqref{1.3}. As Lemma \ref{lem2} implies that $c_3(r,s) \leq 0$, we see that $F_2(x,q) \geq 0$ in this case.

   Subcase 4. $(1+s)x^s-(2-s) \leq 0, 1-s^2 \leq (r-1)(r-2)$ and $1/2 \leq q < q_0$, where $q_0$ is defined by \eqref{3.10}.

   In this case, we set $a=1/2$ and $b=q_0$ in Lemma \ref{lem5}. Note that as $r \leq 3-s$, it follows from this and \eqref{3.10} that when $s \geq -1/2$,
\begin{align*}
   \frac {q_0}{1-q_0} \leq \frac {(2-s)(1-s)}{(1-s^2)} \leq 5.
\end{align*}
    We then deduce that $q_0 \leq 5/6$. Note also that we have  $x^{-s} \geq (1+s)/(2-s) \geq 1/5$ when $s \geq -1/2$.
  Thus, we can take $q_1=5/6$, $x_0=1/5$ in Lemma \ref{lem3} and $q_2=1/2$ in Lemma \ref{lem4} to see that \eqref{2.6} is satisfied by $\alpha_1=-0.0889s, \alpha_2=s(1-s^2)/((r-1)(r-2))$.
    It follows from Lemma \ref{lem5} that $F_2(x,q) \geq 0$ as long as $c_4(r,s) \leq 0$, where $c_4(r,s)$ is given in \eqref{1.3} and  $F_2(x,q) \geq 0$ in this case again follows from Lemma \ref{lem2}.

%%----------------------------------------------------------------------------
\section{Further Discussions}
\label{sec 4} \setcounter{equation}{0}
%%----------------------------------------------------------------------------
    We point out that Theorem \ref{thm2} determines all the pairs $(r,s), r>s$ such that the right-hand side inequality of \eqref{1.5'} holds and all the pairs $(r,s), -1/2 \leq s \leq 1$ such that the left-hand side inequality of \eqref{1.5'} holds. However, less is known for the left-hand side inequality of \eqref{1.5'} when $r>s>1$ or $s<-1/2$. This is partially due to our approach in the proof of Theorem \ref{thm2} relies on showing $F_1(x,q) \leq 0$ (via $F_2(x,q) \geq 0$ ) for $0<x\leq 1, 0<q<1$, where $F_1, F_2$ are defined in \eqref{3.1} and \eqref{3.2}. However, it is easy to see that $F_1(0, q) >0$ when $r>s>1$ and $\displaystyle \lim_{x \rightarrow 0^+}F_2(x,q) < 0$ when $r>2, s<-1$. It also follows from this that in order to show $F_2(x,q) \geq 0$ when $s<-1$, we must have $r \leq 2$. As Lemma \ref{lem1} implies a necessary condition for the left-hand side inequality of \eqref{1.5'} to hold is $r \geq 1,0 \leq r+s \leq 3$, we then deduce that when $s \leq -1$, one can only expect to show $F_2(x,q) \geq 0$ for $1 \leq r \leq 2, s \geq -r \geq -2$.

   On the other hand, though Theorem \ref{thm2} only establishes the validity of the left-hand side inequality of \eqref{1.5'} for $s \geq -1/2$, one can in fact extend the validity of the left-hand side inequality of \eqref{1.5'} for certain $r>s, s <-1/2$ by going through the proof of Theorem \ref{thm2}. This is given in the following
\begin{theorem}
\label{thm3}
  Let $r>s$ and $x_1=\min \{ x_i \}, x_n=\max \{ x_i \}$. The left-hand side inequality of \eqref{1.5'} holds when $(r-1)(r-2) \leq 1-s^2$ or when $-1 < s <-1/2, 2<r<3-s, \displaystyle \max_{1 \leq i \leq 4} \{ c_i(r,s) \} \leq 0$, where $c_i(r,s), 1 \leq i \leq 4$ is defined in \eqref{1.3}. Moreover, in all these cases we have equality holding if and only if $x_1=x_2=\cdots=x_n$.
\end{theorem}
\begin{proof}
   Once again we omit the discussions on the conditions of equality. As in the proof of Theorem  \ref{thm2}, it suffices to prove $F_2(x,q) \geq 0$, where $F_2(x,q)$ is defined in \eqref{3.2}. When $(r-1)(r-2) \leq 1-s^2$, it follows from the expression for $F_2(x,q)/(r-s)$ in \eqref{2.5} that
\begin{align*}
 \frac {F_2(x,q)}{r-s} \geq   \frac {r-1}{r-s}x^{r-2}+\frac {1-s}{r-s}x^{-1-s} -1 \geq   x^{\frac {(r-1)(r-2)-(1-s^2)}{r-s}}-1 \geq 0.
\end{align*}
  When $-1< s <-1/2$, our assertion follows by simply combining the arguments in all the subcases of case 3 in the proof of the left-hand side inequality of \eqref{1.5'} in Section \ref{sec 3}. This completes the proof.
\end{proof}

%%--------------------------------------------------------------------------------------------
%%---------------------------------------------------------------------------------------------

%%-----------------------------------------------------------------------

%%------------------------------------------------------------------------

\end{document}